\documentclass[12pt,oneside,reqno]{amsart}
\usepackage{CJK}
\usepackage[top=1.2in, bottom=1.2in, left=1in, right=1in]{geometry}     
\usepackage{amsmath, amsthm, amsfonts, amssymb, color}
\usepackage{mathrsfs}
\usepackage{extarrows}
\usepackage{latexsym,bm}
\usepackage{graphicx}

      \newtheorem{theorem}{Theorem}[section]
      \newtheorem{lemma}[theorem]{Lemma}
      \newtheorem{corollary}[theorem]{Corollary}
      
      \theoremstyle{definition}
      
      \theoremstyle{remark}
      \newtheorem{remark}[theorem]{Remark}

      \newcommand{\R}{{\mathbb R}}
      
      \newcommand{\<}{\langle}
      \newcommand{\Y}{\rangle}
      \newcommand{\z}{\left}
      \newcommand{\y}{\right}
      \newcommand{\E}{\mathbf{E}}
      \newcommand{\Pro}{\mathbf{P}}
      \newcommand{\all}{\mathbf{E}\int_0^T dt}

      \makeatletter
      \def\@setcopyright{}
      \def\serieslogo@{}
       
      \newcommand{\Rmnum}[1]{\expandafter\@slowromancap\romannumeral #1@}
      \makeatother

\begin{document}
\title{Compact support property of SuperBrownian  Motion in random environments}
\author{Guohuan Zhao}
\maketitle 
\begin{abstract}
In this paper, we prove the compact support property for a class of nonlinear SPDE including the equation that the density of one-dimensional Super-Brownian Motion in random environment satisfies.
\end{abstract}
\section{Introduction}
It's well know that the one-dimensional density of classic Super-Brownian  Motion satisfies the following SPDE, 
\begin{equation}\label{eq1}\partial_t u(t,x)=\Delta u(t,x)+\sqrt{u(t, x)}\dot{B}(t,x)\end{equation}
 Here we replace $\frac{1}{2}\Delta$ with $\Delta$ for simple and $\dot{B}(t,x)$ is time-space white noise. The above equation can be written as following form: 
$$\partial_t u(t,x)=\Delta u(t,x)+\sqrt{u(t,x)}\varphi_k(x) \dot{\beta}^k_t$$
where $\{\varphi_k\}$ is an orthlnormal basis of $L^2(\R)$ and $\{\beta_t^k\}$ is a sequence of independent Brownian Motions. The solutions to this equation has \textbf{compact support property}, roughly speaking, if the initial date $u(0,x)$ have compact support, then for all $t>0$, $u(t,\cdot)$ has compact support almost surely. In \cite{mueller1992compact}, the authors proved the compact support property for the solutions of a large class of SPDEs including \eqref{eq1}. Later Krylov given a simpler proof in \cite{krylov1997result} by using his $L^p$ theory.  

One the other hand, from 1990's many experts started to study superprocesses in random environments. In \cite{mytnik1996superprocesses}, Mytnik introduced models of  superprocesses in random environments. We give a brief description below: 

Let $\{\xi_k(x), k\in\mathbb{N}\}$ be a sequence of independent identically distributed random fields on $\R^d$ satisfying:
\begin{equation}
  \E\xi_k(x)=0,\quad \E\xi_k(x)\xi_k(y)=g(x,y), \quad\sup_{x}\E|\xi_k(x)|^3<\infty, \quad x\in \mathbb{R}^d,k\in\mathbb{N}. 
\end{equation}
Where $g$ is the covariance function which satisfies $\sup_{x,y\in \R}|g(x,y)|\leq C<\infty$, $g(x,\cdot)\in C_0(\R^d)$ $(\forall x\in\R^d)$, we further assume $g\in C^2_b$ in this paper.
$\{\xi_k(x), k\in\mathbb{N}\}$ serves as the random environments. For each fixed $n\in \mathbb{N}$, consider a particle system in which there are $K_n\ge 1$ particles located in $\R^d$, each of them moves independent  as a copy of Brownian motion(with generator $\Delta$) until time $t=1/n$. Given $\{\xi_k(x), k\in\mathbb{N}\}$, at time $\frac{1}{n}$, each particle split into two particles with probability $\frac{1}{2}+\frac{1}{2\sqrt{n}}[(-\sqrt{n})\vee \xi_1(x)\wedge(\sqrt{n})]$ or dies with probability $\frac{1}{2}-\frac{1}{2\sqrt{n}}[(-\sqrt{n})\vee \xi_1(x)\wedge(\sqrt{n})]$. The new particles then moves in space independently as  Brownian motions(with generator $\Delta$)  starting at their place of birth, 
during the time interval $[1/n, 2/n)$. In general,  at time $\frac{i}{n}$, each surviving particle split into two particles with probability $\frac{1}{2}+\frac{1}{2\sqrt{n}}[(-\sqrt{n})\vee \xi_i(x)\wedge(\sqrt{n})]$ or dies with probability $\frac{1}{2}-\frac{1}{2\sqrt{n}}[(-\sqrt{n})\vee \xi_i(x)\wedge(\sqrt{n})]$, and in the time interval $[i/n, (i+1)/n]$ particles independently according to Brownian motions(with generator $\Delta$).
 Let $X^n_t$ be the measure-valued Markov process, defined as
$$X^n_t(B)=\frac{\mbox{number of particles in }B\mbox{ at time }t}{n}.$$
where $B\in{\mathcal B}(\R^d)$ are Borel sets in $\R^d$.
Let $C^k_b(\R^d)$ (respectively $C^\infty_b(\R^d))$ denote the collection of all bounded continuous functions on $\R^d$ with bounded continuous derivatives up to order $k$ (respectively with bounded derivatives of all orders). For all bounded measurable $f$, let $\mu(f)=\langle f,\mu\rangle$ denote the integral of
$f$ with respect to the measure $\mu$ on $\R^d$. For any measurable functions $f,h$ on $\R^d$, let $f\otimes h$ denote an $\R^{2d}$ valued function defined by  $(f\otimes h)(x,y):=f(x)h(y)$, $x,y\in\R^d$.
It was proved in \cite{mytnik2007local} that if $X^n_0$ converge weakly to a finite measure $\mu$
 on $\R^d$, then the processes $X^n=\{X^n_t, t\ge 0\}$ converges weakly to a  measure-valued process
 $X=\{X_t,t\ge 0\}$, where $X$ is the unique solution to the following martingale problem:

\begin{equation}\label{MP}
\mbox{(MP)}:\left\{\begin{array}{rl}&\displaystyle\mbox{for all }f\in C^2_b,\quad M_t^f=\langle f, X_t\rangle-\langle f, \mu\rangle-\int_0^t\langle \Delta f,  X_s\rangle\\
&\displaystyle\mbox{ is a continuous square-integrable martingale with quadratic variation }\\
&\displaystyle\langle M^f \rangle_t= \int_0^t \langle f^2, X_s\rangle ds+ \int_0^t \langle g\cdot f\otimes f, X_s\otimes X_s\rangle ds.
\end{array}
\right.\\
\end{equation}

However, there only very few works about the properties of these processes. As far as to my knowledge, the most interesting work is \cite{mytnik2007local}, in which the authors studies the local extinction property. Under some assumption, it shows the super-Brownian motion in random environments will extinct locally in any dimension which is quite different from the classic case(see Theorem 1.1, 1.2 of \cite{mytnik2007local} for more details). 

In this paper, we show the density of one-dimensional super-Brownian motion in random environments satisfies: 
$$\partial_tu=\Delta u+\sqrt{u}\varphi_k \dot{\beta}^k_t+uh_k\dot{w}_t^k. $$
Here $\{w_t^k\}$ and $\{\beta_t^k\}$ are two sequences of independent standard Brownian motions. In order to do that we give the second order moment formula of superprocesses in Section 2. In Section 3, using Krylov's $L^p$ theory for SPDEs (see \cite{krylov1999analytic}), we prove the compact support property for more general stochastic partial differential equations under some reasonable assumptions. 

\section{Moment Formula}
In this section, we give the first and second order moment formula for $X_t$. 

The following Lemma is a simple application of Stone-Weierstrass theorem. 
\begin{lemma}\label{apx1}
For any $f(x,y)\in C_b(\R^{d_1+d_2})\,\,(x\in \R^{d_1}, y\in \R^{d_2})$, there exist a sequence of smooth functions $\{f_n(x,y)\}$ with form $f_n(x,y)=\sum_{i=1}^{n}\phi^n_i(x)\psi^n_i(y)\}$ such that $f_n\rightarrow f $ uniformly on any compact subset in $\R^{d_1+d_2}$ and $\|f_n\|_\infty\leq \|f\|_\infty$. 
\end{lemma}
In order to get the second moment formula, we need the following generalization of Lemma \ref{apx1}. 

\begin{lemma}\label{apx2}
For any $f(x,y)\in C^{k_1,k_2}_b(\R^{d_1+d_2})~~(x\in \R^{d_1}, y\in \R^{d_2})$, there exist a sequence of smooth function $\{f_n(x,y)\}$ with form $f_n(x,y)=\sum_{i=1}^{n}\phi^n_i(x)\psi^n_i(y)\}$ such that for any $\alpha=(\alpha_1,\cdots,\alpha_d)\in \mathbb{N}^{d_1}$, $\beta=(\beta_1,\cdots,\beta_d)\in \mathbb{N}^{d_2}$ and $|\alpha|=\alpha_1+\cdots+\alpha_d\leq k_1$, $|\beta|=\beta_1+\cdots+\beta_d\leq k_2$. $\partial_x^\alpha\partial_y^\beta f_n\rightarrow \partial_x^\alpha\partial_y^\beta f$ uniformly on any compact subset and $\|f_n\|_{C^{k_1, k_2}}\leq C\|f\|_{C^{k_1,k_2}}$. 
\end{lemma}
\begin{proof}
To keep the proof simple, we assume $k_i, d_i=1$. Let $\varphi(x)\in C^\infty_c(\R)$ and $\varphi(x)=1$ if $|x|\leq 1/2$, $\varphi(x)=0$ if $|x|\geq 1$. Fixed $R>0$, let 
$f_R(x,y)=f(x,y)\varphi(\frac{x}{R})\varphi(\frac{y}{R})$. Using standard diagonal argument, in order to prove the Lemma, we only need to show there exists a sequence of smooth functions $f_{R, n}(x,y)=\sum_{i=1}^{n}\phi^n_i(x)\psi^n_i(y)$ such that $\|f_{R,n}-f_R\|_{C^1}\rightarrow 0~~(n\rightarrow 0)$. Now define 
$$f^m_R(x,y)=\int_{\R^2} m^{2} \varphi(m\xi)\varphi(m\eta)f_R(x-\xi, y-\eta) d\xi d\eta\in C^\infty_c(\R^2).$$
Since $f_R\in C^1_c$, we have 
\begin{align}\label{C1est1}\lim_{m\rightarrow 0}\|f^m_R-f_R\|_{C^1}=0.\end{align}
On the other hand, we have 
$$f^m_R(x,y)=\int_{-R}^y\int_{-R}^x \partial_{x}\partial_y f^m_R(\xi, \eta) d\xi d\eta, $$ 
by Lemma \ref{apx1}, there exist a sequence of smooth functions $g_{R} ^{m,n}=\sum_{i=1}^{n}\phi^{m,n}_i(x)\psi^{m,n}_i(y)$ such that $g_{R} ^{m,n}(x,y)\rightarrow \partial_x\partial_y f^m_R$ uniformly as $n\rightarrow \infty$. Let $f^{m,n}_R=\int_{-R}^y\int_{-R}^x g^{m,n}_R(\xi, \eta) d\xi d\eta$, then 
\begin{align}\label{C1est2}\lim_{n\rightarrow \infty}\|f^{m,n}_R-f^m_R\|_{C^1}=0. \end{align} 
Combining  \eqref{C1est1} and \eqref{C1est2}, using standard  diagonal argument, we can find a sequence of function $f_{R,n}=\sum_{i=1}^{n}\phi^n_i(x)\psi^n_i(y)\rightarrow f_R$ in $C^1$.    
\end{proof}

Before giving the second moment formula, we first prove an estimate for $\E_\mu\<f,X_t\Y^2$. 
\begin{lemma}\label{le0}
$$\E_\mu \<f,X_t\Y=\<P_tf, \mu\Y;$$
$$\E_\mu\<f,X_t\Y^2\leq \left\{\<P_tf,\mu\Y^2+\int_0^t \<\mu, P_s[(P_{t-s}f)^2]\Y dr\right\}\exp(^{\|g\|_\infty t}).$$
Where $\{P_t\}$ is the semigroup whose generator is $\Delta$. 
\end{lemma}
\begin{proof}
Just as the proof of Proposition II 5.7 of \cite{perkins2002part}, if $\phi_t(x)\in C^{1,2}(\R_+\times \R^d)$, we can prove 
$$\<\phi_t,X_t\Y-\<\phi_0,\mu\Y-\int_0^t\<\dot{\phi}_s+\Delta\phi_ s,X_s\Y ds$$
is a martingale with quadratic variation 
$\<M\Y_t=\int_0^t\<\phi_s^2,X_s\Y ds+\int_0^t\<g\cdot\phi_s\otimes\phi_s, X_s\otimes X_s\Y ds$. 
In addition, there exists a martingale measure $M(dt,dx)$ such that for any $f\in\mathcal{B}(\R^d)$, 
$$\<f,X_t\Y=\<\mu, P_tf\Y+\int_0^t\int_{\R^d}P_{t-s}f(x)M(ds,dx)$$
$\left\{ \int_0^s p_{t-r}f(x)M(dr,dx)\right\}_{s\leq t}$ is a martingale from time $0$ to $t$ with 
\begin{align*}
\<\int_0^\cdot P_{t-r}f(x)M(dr,dx)\Y_s=&\int_0^s\int_{\R^d} (P_{t-r}f(x))^2 X_r(dx)dr\\
+&\int_0^s \int_{\R^{2d}}g(x,y)P_{t-r}f(x)P_{t-r}f(y)X_r(dx)X_r(dy)dr
\end{align*}
Hence 
$$\E_\mu \<f,X_t\Y=\<P_tf, \mu\Y,$$
\begin{align*}
\E_\mu[\<P_{t-s}f,X_s\Y^2]= &[\<P_tf,\mu\Y]^2+\E_\mu \int_0^s\int_{\R^d} (P_{t-r}f(x))^2 X_r(dx)dr\\
+&\E_\mu \int_0^s \int_{\R^{2d}}g(x,y)P_{t-r}f(x)P_{t-r}f(y)X_r(dx)X_r(dy)dr\\
\leq& [\<P_tf,\mu\Y]^2+\int_0^s \<\mu,P_r[(P_{t-r}f)^2]\Y dr +\|g\|_{\infty}\int_0^s\E_\mu [\<P_{t-r}f,X_r\Y^2 ]ds
\end{align*}
By Gronwall's Inequality, we obtain 
$$E_\mu[\<f,X_t\Y^2]\leq \left\{[\<P_tf,\mu\Y]^2+\int_0^t \<\mu,P_r[(P_{t-s}f)^2]\Y ds \right\}\exp({\|g\|_{\infty}t})$$
\end{proof}

Now we are in a position to prove the second moment formula:
\begin{theorem}
Let $Q_t=P_t \otimes P_t$ and $Q_t^{g}$ be the semigroup generated by ${\Delta}+{g}$.
 Then for all $\phi\in C_b(\R^d)$,
$$\E_\mu(X_t(\phi)^2)=\<Q_t^{g}(\phi\otimes \phi),\mu\otimes\mu\Y+\int_0^t\<P_s(\mathfrak{\pi} Q_{t-s}^{g}(\phi\otimes \phi))\mathrm{d}s,\mu\Y$$
Here 
$$(\pi f)(x)=f(x,x);~~~f\in \mathcal{B}(\R^{2d}), \, x\in \R^d. $$
\end{theorem}

\begin{proof}
We assume $\mu=\delta_0$ for simple. The proof for general $\mu\in \mathcal{M}_F(\mathbb{R}^d)$ is similar. 

$\forall \phi_s, \, \psi_s \in C^{1,2}_b(\R_+\times\R^d)$, using Ito's Formula,
\begin{displaymath}
\begin{aligned}
\mathrm{d}(X_s(\phi_s)X_s(\psi_s))&=X_s(\phi_s)\mathrm{d}(M_s^\psi+V_s^\psi)+X_s(\psi_s)\mathrm{d}(M_s^\phi+V_s^\phi)+\mathrm{d}\langle M^\phi,M^\psi\rangle_s\\
\end{aligned}
\end{displaymath}
Where 
$$V^\phi_s=\langle \phi, \mu\rangle+\int_0^s\langle \dot{\phi_r}+\Delta \phi_r,  X_r\rangle dr,~~~~~V^\psi_s=\langle \psi, \mu\rangle+\int_0^s\langle \dot{\psi_r}+\Delta \psi_r,  X_r\rangle dr. $$
$$M_s^\phi=\langle \phi_s, X_s\rangle-V^\phi_s, ~~~~M^{\psi}_s=\langle \psi, X_s\rangle-V^\psi_s;$$
taking expectation,
\begin{displaymath}
\begin{aligned}
\E_{\delta_0}(X_t(\phi_t)X_t(\psi_t))&=\phi_0(0)\psi_0(0)+\E_{\delta_0}\int_0^t X_s(\phi_s)\mathrm{d} V_s^\psi+\E_{\delta_0}\int_0^t X_s(\psi_s)\mathrm{d} V_s^\phi+\E_{\delta_0}\langle M^\phi,M^\psi\rangle_t\\
&=\phi_0(0)\psi_0(0)+\E_{\delta_0}\int_0^t [X_s(\phi_s)X_s(\dot{\psi}_s+{\Delta}\psi_s)+X_s(\psi_s)X_s(\dot{\phi}_s+{\Delta}\phi_s)]ds\\
&~~~~+X_s\otimes X_s(g\cdot \phi_s\otimes \psi_s)\mathrm{d}s+\E_{\delta_0}\int_0^t X_s(\phi_s \psi_s)\mathrm{d}s\\
&=\phi_0(0)\psi_0(0)+\E_{\delta_0}\int_0^t X_s\otimes X_s(\partial_s(\phi_s\otimes \psi_s)+({\Delta}+{g})(\phi_s\otimes\psi_s))\mathrm{d}s\\
&~~~~+\E_{\delta_0}\int_0^t X_s(\mathfrak{\pi} \phi_s\otimes\psi_s)\mathrm{d}s\\
\end{aligned}
\end{displaymath}
By linearity, identity 
\begin{align}\label{SecondMoment}
\E_{\delta_0}(X_t\otimes
X_t (f(t)))=f(0,0,0)+\E_{\delta_0}\int_0^tX_s\otimes X_s(\dot{f}(s)+{\Delta}f(s)+gf(s))\mathrm{d}s+\E_{\delta_0}\int_0^tX_s(\mathfrak{\pi} f(s))\mathrm{d}s
\end{align}
holds for all $f(t,x,y)=\sum_{i=1}^{n}\phi_i(t,x)\psi_i(t,y)$. 

$\forall f(s,x,y)\in C^{1,2}_b(\R_+\times\R^{2d})$.  By Lemma \ref{apx2}, we can find functions $f^n(s,x,y)=\sum_{i=1}^n \phi_i^n(s,x)\psi_i^n(s,y)$  such that for any fixed $R>0$, 
$$\lim_{n\rightarrow 0}\|f^n-f\|_{C^{1,2}(Q_R)}=0$$ and $\|f^n\|_{C^{1,2}_b}\leq C\|f\|_{C^{1,2}_b}$. Here $Q_R= [0, R]\times[-R,R]^{2d}$ and $C$ is independent with $n$. 

Let $\delta^n(s,x,y)=f(s,x,y)-f^n(s,x,y)$, then $\lim_{n\rightarrow 0}\|\delta^n\|_{C^{1,2}(Q_R)}=0$ and $\|\delta^n\|_{C^{1,2}_b}\leq C\|f\|_{C^{1,2}_b}$, here $C$ is independent with $n$. Define $I_R(x)=I_{[-R,R]^d}(x)$, $I'_R(x)=1-I_R(x)$. we get
$$|\delta^n(s,x,y)|\leq C (I'_R(x) I'_R(y)+I'_R(x) I_R(y)+I_R(x) I'_R(y))+\|\delta^n\|_{C^{1,2}(Q_R)} I_R(x) I_R(y).$$ 
Using Lemma \ref{le0}, 
\begin{equation}\label{EdeltaR}
\begin{aligned}
\E_{\delta_0} X_s\otimes X_s (|\delta^n(s)|)
\leq&  C \E_{\delta_0}\left( X_s(I_R) X_s(I'_R) + X_s(I'_R)^2\right)+c_n \E_{\delta_0}X_s(I_R)^2\\
\leq& C \left\{\E_{\delta_0} X_s(I'_R)^2 + \left[\E_{\delta_0}X_s(I'_R)^2\right]^{\frac{1}{2}}\left[\E_{\delta_0}X_s(I_R)^2\right]^{\frac{1}{2}} \right\}\\
&+c_n \E_{\delta_0}X_s(I_R)^2.
\end{aligned}
\end{equation}
Here $c_n=\|\delta^n\|_{L^\infty(Q_R)}\rightarrow 0~~(n\rightarrow \infty)$.  

$\E_{\delta_0}X_s(I_R)^2 \leq \E_{\delta_0}X_s(1)^2 \leq C $ and again by Lemma \ref{le0},
\begin{align}\label{I'}
\E_{\delta_0}X_s(I'_R)^2 \leq C\left\{[P_s\I'_R]^2(0)+\int_0^s P_r[(P_{s-r}I'_R)^2](0)dr\right\}\end{align}
For any $\epsilon>0$, choose $R$ so large, such that $\frac{1}{(\sqrt{2\pi t})^d}\int_{|y|>\frac{R}{2}}e^{-|y|^2/2t}dy<\epsilon^2$.
Then
  \begin{align}\label{I'1}\sup_{s\leq t}[P_s\I'_R]^2(0)\leq \z(\frac{1}{\sqrt{2\pi t}}\int_{|x|>R}e^{-x^2/2t}dx \y)^2<\epsilon^2 . \end{align}

$$[(P_{s-r}I'_R)^2](x)\leq \left(I_{\frac{R}{2}}(x)\int_{|y|>R}e^{|x-y|^2/2(s-r)}dy+I'_{\frac{R}{2}}(x)\right)^2\leq C\epsilon^2 I_{\frac{R}{2}}(x)+CI'_{\frac{R}{2}}(x)$$
Hence
\begin{align}\label{I'2}\int_0^s P_r[(P_{s-r}I'_R)^2](0)dr\leq C \epsilon^2 s + C \int_0^s\frac{1}{\sqrt{2\pi r}}\int_{|y|>\frac{R}{2}}e^{|y|^2/2r} dy dr\leq C\epsilon^2\end{align}
So by \eqref{EdeltaR}, \eqref{I'}, \eqref{I'1}, \eqref{I'2}, we have
\begin{equation}\label{epsilon1}
\lim_{n\rightarrow \infty}\sup_{s\leq t}\E_{\delta_0} X_s\otimes X_s (|\delta^n(s)|)=0. 
\end{equation} By the same argument we can prove\\
\begin{equation}\label{epsilon2}
\begin{aligned}
  &\lim_{n\rightarrow \infty}\left|\E_{\delta_0}\int_0^t X_s\otimes X_s(\partial_s f^n(s)+({\Delta}+g )f^n(s))\mathrm{d}s-\E_{\delta_0}\int_0^t X_s\otimes X_s(\partial_s f(s)+({\Delta}+g )f(s))\mathrm{d}s\right|\\
  \leq &\lim_{n\rightarrow \infty} \int_0^t \E_{\delta_0} X_s\otimes X_s\left(\left|\partial_s\delta^n(s)+(\Delta+g)\delta^n(s)\right|\right)ds \\
  \leq &C\|f\|_{C^{1,2}([0,t]\times\mathbb{R}^2)} \lim_{R\rightarrow \infty }\int_0^t \E_{\delta_0}\left( X_s(I_R) X_s(I'_R) + X_s(I'_R)^2\right)ds\\
  &+\lim_{R\rightarrow}\lim_{n\rightarrow \infty} \|\delta^n\|_{C^{1,2}(Q_R)} \int_0^t \E_{\delta_0} X_s(I_R)^2ds\\
 =&0
\end{aligned}
\end{equation}
Similarly, 
\begin{equation}\label{epsilon3}
\lim_{n\rightarrow \infty}\left|\E_{\delta_0}\int_0^tX_s(\mathfrak{\pi} f^n(s))\mathrm{d}s-\E_{\delta_0}\int_0^t X_s(\mathfrak{\pi} f(s))\mathrm{d}s\right|=0.
\end{equation}
So we obtain \eqref{SecondMoment} holds for all $f\in C^{1,2}_b(\R_+\times\mathbb{R}^{2d})$.

Suppose $\phi\in C^\infty_c$, let $f_s=Q_{t-s}^{g}\phi\otimes \phi$ (define $f_s=f_t$ if $s>t$), then $f\in C^{1,2}_b(\R_+\times\mathbb{R}^2)$. Hence for all $\phi\in C_c^\infty(\mathbb{R})$, we have the formula
  $$\E_{\delta_0}(X_t(\phi)^2)=(Q_t^{g}(\phi\otimes \phi))(0,0)+\int_0^t[P_s(\mathfrak{\pi} Q_{t-s}^{g}(\phi\otimes \phi))](0)\mathrm{d}s. $$
A simple approximation argument shows for all $\phi\in C_b(\mathbb{R}^d)$

$$\E_{\delta_0}(X_t(\phi)^2)=\<Q_t^{g}(\phi\otimes \phi),\delta_0\otimes\delta_0\Y+\int_0^t\<P_s(\mathfrak{\pi} Q_{t-s}^{g}(\phi\otimes \phi))\mathrm{d}s,\delta_0\Y$$\end{proof}

\begin{remark}
Indeed, we can also using conditional Laplace transform introduced by \cite{mytnik2007local} to get the same formula. However, the proof presented here is more elementary.
\end{remark}

\section{Compact Property}
In this section, we first use the moment formula to get the equation that the density of Super-Brownian  motion satisfies and then prove the compact support property for a class of parabolic SPDEs.

\begin{lemma}
Suppose $\mu(\R)<\infty$, then the density of the 1-$d$ Super-Brownian  Motion in random environments satisfies the following SPDE:
$$\partial_tu=\Delta u+\sqrt{u}\varphi_k \dot{\beta}^k_t+uh_k\dot{w}_t^k$$
in weak sense, which means for any $\phi\in C_c^\infty(\R)$ we have
\begin{equation} 
\begin{split}
\int_\R \phi(x)u(t,x)dx=&\int_\R \phi(x)u(0,x)+\int_0^t\int_\R u(s,x)\phi''(x)dxds\\
&+\sum_{k=1}^\infty \int_0^t\int_\R \phi(x)\varphi_k(x)\sqrt{u(s,x)}dxd\beta^k_s+\sum_{k=1}^\infty \int_0^t\int_\R \phi(x)h_k(x)u(s,x)dxdw^k_s. 
\end{split}
\end{equation}
Where $\beta^k$, $w^k$ are independent Brownian Motion. 
\end{lemma}
\begin{proof}
By the second moment formula, we have 
$$\E_\mu(\<X_t,\phi\Y\<X_t,\psi\Y)=(\mu\otimes\mu)(Q_t^g(\phi\otimes\psi))+\int_0^t \mu[P_s(\pi Q_{t-s}^g(\phi\otimes\psi))]$$
Denote $q_t^g$ be the density of $Q_t^g$. Using the above equation,  
\begin{equation}\label{L2}
\begin{split}
&\E_\mu(\<X_t, p({\epsilon}, x-\cdot)\Y\<X_t,p({\epsilon '}, x-\cdot)\Y)\\
=&\int_{\R^2} \mu(y_1)\mu(y_2)\int_{\R^2}q_t^g((y_1,y_2),(z_1,z_2)) p(\epsilon,x-z_1) p(\epsilon ', x-z_2))dz_1dz_2\\ 
+&\int_0^t ds\int_\R \mu(dw) \int_{\R^3}p(s,w-y)q_{t-s}^g((y,y),(z_1,z_2)) p(\epsilon,x-z_1) p(\epsilon ', x-z_2))dz_1dz_2dy\\
=&I(t,x)+II(t,x)
\end{split}
\end{equation}
It's not hard to prove that 
\begin{equation}\label{estimateI}
\begin{split}
&\int_0^T dt \int_\R I(t,x) dx\\
=&\int_0^Tdt\int_\R dx\int_{\R^2} \mu(dy_1)\mu(dy_2)\int_{\R^2}q_t^g((y_1,y_2),(z_1,z_2)) p(\epsilon,x-z_1) p(\epsilon ', x-z_2))dz_1dz_2\\
\xlongrightarrow{\epsilon,\epsilon'\rightarrow 0}&\int_0^Tdt\int_\R dx\int_{\R^2} \mu(dy_1)\mu(dy_2)\int_{\R^2}q_t^g((y_1,y_2),(x,x)) \\
\leq& C \int_0^Tdt \int_{\R^2}\mu(dy_1)\mu(dy_2)\int_{\R}p(t, x-y_1)p(t,x-y_2)dx\\
\leq & C\mu(\R)^2 \int_0^T \frac{1}{\sqrt{t}}dt<\infty 
\end{split}
\end{equation}
By the same argument, 
\begin{equation}\label{estimateII}
\begin{split}
\int_0^T dt \int_\R II(t,x) dx\xlongrightarrow{\epsilon,\epsilon' \rightarrow 0} &\int_0^T dt \int_0^t ds\int_\R \mu(dw) \int_\R dx\int_{\R^3}p(s,w-y)q_{t-s}^g((y,y),(x,x))dy\\
\leq & C \mu(\R) \int_0^T dt \int_0^t \frac{1}{\sqrt{t-s}}ds<\infty 
\end{split}
\end{equation}
By \eqref{L2}, \eqref{estimateI}, \eqref{estimateII},  we get if $\mu(\R)<\infty$, $\{\< p(\epsilon, x-\cdot),X_t\Y\}$ is a Cauchy sequence in $L^2 (\Omega \times [0,T]\times\R)$, define 
$$u_t(x)=\lim_{\epsilon\rightarrow 0}\<p(\epsilon, x-\cdot),X_t\Y$$
we have 
$$M^\phi_t=(u_t,\phi)-(u_0,\phi)-\int_0^t(u_s,\Delta\phi)ds$$ is a martingale with quadratic variation $$\<M^\phi\Y_t=\int_0^t(\phi^2,u_s)ds+\int_0^t ds \int_{\R^2}\sum_{k} h_k(x)h_k(y) \phi(x)\phi(y)dxdy$$
By martingale representation theorem, there exist independent Brownian sheet $B(t,x)$ and time-white, space-colored Guassian noise $W(t,x)$ with $\E(W(t,x)W(s,y))=(s\wedge t) g(x,y)$ such that 
$$M_t^\phi=\int_0^t\int_\R\phi(x)\sqrt{u_s(x)}B(ds,dx)+\int_0^t \int_\R \phi(x)u_s(x)W(ds,x)dx$$
The rightside of above equation can be written as 
 $$\int_0^t (\phi,\sqrt{u}\varphi_k)d\beta_{s}^k+\int_0^t (\phi,u_sh_k)dw_s^k$$
where $\{\varphi_k\}$ is a orthonormal basis of $L^2(\R)$, $\sum_k h_k(x)h_k(y)=g(x,y)$ and $\beta^k_t, ~~w^k_t$ are independent Brownian motions. 
Hence 
$$(u_t,\phi)=(u_0,\phi)+\int_0^t(u_s,\phi'')ds+\int_0^t(\sqrt{u_s}\varphi_k,\phi)d\beta_s^k+\int_0^t (u_sh_k,\phi)dw^k_s$$
\end{proof}

Now we begin to consider the compact support property of following parabolic SPDEs:
$$
\partial_tu=\Delta u+u^\gamma\varphi_k \dot{\beta_t}^k+uh_k\dot{w_t}^k
$$
\begin{equation}\label{eq2}
(u_t,\phi)=(u_0,\phi)+\int_0^t(u_s,\phi'')ds+\int_0^t(u_s^\gamma \varphi_k,\phi)d\beta_s^k+\int_0^t (u_sh_k,\phi)dw^k_s
\end{equation}
Here $\gamma\in [1/2,1)$, $\{\varphi_k\}$ is the standard orthonormal basis of $L^2(\R)$, $\{h_k\}$ satisfies \\
$$\sup_{x\in\R}\sum_k h^2_k(x)<\infty. $$

Define 
$$C_{tem}=\z\{f\in C(\R): \int_{\R}e^{-\lambda|x|}|f(x)|dx<\infty; \forall \lambda>0\y\}$$

The next lemma is standard. 
\begin{lemma}\label{le1}
If $u_t(x)\in C_{tem}$ is the weak solution to Equation (\ref{eq2}) with initial data $u_0\in C_{tem}$, then $u_t(x)$ satisfies the following equation:
\begin{equation}\label{mildform}
\begin{aligned}
u_t(x)=&p_t*u_0(x)+\int_0^t \z[\int_{\R} p_{t-s}(x-y) u^\gamma_s(y)\varphi_k(y)dy\y]d\beta^k_s\\
&+\int_0^t \z[\int_{\R} p_{t-s}(x-y)u_s(y)h_k(y)dy\y]dw^k_s
\end{aligned}
\end{equation}
\end{lemma}

\begin{lemma}\label{le2}
Suppose $u_0\in C_{tem}^+$, there exists an $\{\mathfrak{B}_t\}$-space-time white noise $\dot{B}(t,x)$, an independent $\{\mathfrak{W}_t\}$-time white space colored noise $\dot{W}(t,x)$ and a $C(\R_+; C^+_{tem})\cap C(\R_+\times \R)$ solution $u(t, \cdot) \in \mathfrak{F}_t=\mathfrak{B}_t\vee\mathfrak{W}_t$ to \eqref{eq2} on a suitable probability space with filtration $(\Omega, \mathfrak{F}, \mathfrak{F_t}, \Pro)$. What's more, for any $\lambda>0$, 
\begin{align}\label{Pmoment}
\sup_{t\leq T}\E \int_{\R} |u_t(x)|^pe^{-\lambda |x|}dx\leq C(T,\lambda)\z\{1+ \int_\R u^p_0(x)e^{-\lambda|x|}dx\y\} \end{align}
\end{lemma}
\begin{proof}
The proof for existence of $C_{tem}^+$ solution to \eqref{eq2} is similar with Theorem 2.6 in \cite{shiga1994two}, so we only prove \eqref{Pmoment} here. 

Taking $p$'s power in both side of \eqref{mildform} then taking expectation, using BDG inequality and Minkowski inequality, we obtain 
\begin{align*}
\E|u_t(x)|^p\leq& C\Big\{|p_t*u_0|^p+ \E\Big[\int_0^t ds \int_{\R}p^2_{t-s}(x-y) u^{2\gamma}_s(y) dy\Big]^{p/2}\\
&+ \E \Big[\int_0^t ds \sum_{k}\Big(\int_{\R} p_{t-s}(x-y) u_s(y)h_k(y)dy\Big)^2 \Big]^{p/2}\Big\}\\
\leq &C\Big\{ |p_t*u_0|^p + \E\Big[\int_0^t (t-s)^{-1/2}ds \int_{\R} (t-s)^{1/2}p^2_{t-s}(x-y)(1+u^2_s(y))dy\Big]^{p/2} \\
&+ \int_0^t ds \int_{\R} p_{t-s}(x-y) \E u^p_s(y)dy\Big\}\\
\leq &C \Big\{|p_t*u_0|^p + \int_0^t (t-s)^{-1/2}ds 
\int_{\R} (t-s)^{1/2}p^2_{t-s}(x-y)(1+(\E u^p_s(y))dy\Big]\\
&+\int_0^t ds \int_{\R} p_{t-s}(x-y) \E u^p_s(y)dy\Big\}
\end{align*}
Hence, for any $t\leq T$ we have 
\begin{align*}
\int_{\R}e^{-\lambda |x|}\E|u_t(x)|^p dx\leq& C \Big\{1+ \int_{\R} u^p_0(y)dy \int_{\R}p_t(x-y)e^{-\lambda|x|}dx  \\
&+\int_0^t (t-s)^{-1/2}ds \int_{\R} \E u^p_s(y)dy \int_{\R} (t-s)^{1/2}p^2_{t-s}(x-y)e^{-\lambda |x|}dx\\
&+ \int_0^t ds \int_{\R} \E u^p_s(y)dy
\int_{\R}  p_{t-s}(x-y) e^{-\lambda |x|} dx \Big\}\\
\leq & C\Big\{ 1+\int_{\R} u_0^p(x)e^{-\lambda|x|}dx+ \int_0^t \frac{ds}{\sqrt{t-s}} \int_{\R} e^{-\lambda |y|} u_s^p(y)dy \Big\}
\end{align*}
In the last inequality, we use the element inequality: $\sup_{t\leq T} e^{\lambda|y|}\int_{\R} p_t(x-y) e^{-\lambda |x|} dx\leq C$. 
Denote $f(t)=\sup_{s\leq t} \int_{\R}e^{-\lambda |x|}\E|u_s(x)|^p dx$, $A=1+\int_{\R} u^p_0(x)e^{-\lambda|x|}dx$ then, 
\begin{align*}
f(t)\leq CA+C\int_0^t f(s)\frac{ds}{\sqrt{t-s}}\leq & CA+ C\int_0^t \frac{ds}{\sqrt{t-s}} \Big(CA+C\int_0^s f(r) \frac{dr}{\sqrt{s-r}}\Big)\\
\leq & CA+C\int_{0}^t f(r)dr\int_r^t \frac{ds}{\sqrt{(t-s)(s-r)}}\\
\leq & CA+C\int_0^t f(s)ds
\end{align*}
Using Gronwall's inequality, we obtain \eqref{Pmoment}. 
\end{proof}

\begin{corollary}
Suppose $u\in C(\R_+;C^+_{tem})$ is a solution to \eqref{eq2} with $u_0(x)=0~~(x\geq 0)$ then for any $T>0$, 
$$a_p(T)\triangleq\sup_{n\in \mathbb{N}}\E \int_0^Tdt\int_n^{n+1}u_t^p(x)dx<\infty$$
\end{corollary}
\begin{proof}
For any $n\in \mathbb{N}$, let $v_t(x)=u_t(n+x)\in C([0,T],C_{tem}^+)$, $v$ satisfies the equation
$$\partial_t v_t(x)=\Delta v_t(x)+v^{\gamma}_t(x)\varphi_k(n+x)\dot{\beta_t}^k+v_t(x)h_k(n+x)\dot{w_t}^k$$
Since $\{\varphi_k(n+\cdot)\}$ is again the orthnormal basis of $L^2(\R)$, $\{h_k(n+\cdot)\}$ satisfies the some condition with $\{h_k\}$. By Lemma (\ref{le2}) we have
\begin{align*}
E\int_0^T\int_0^1v^p_t(x)dxdt&\leq C(T)\z\{1+\int_{-\infty}^{0} e^{-|x|}u_0(n+x)dx\y\}\\
&\leq C(T)\z\{1+\int_{-\infty}^{-n} e^{-|x|}u_0(x)dx\y\}\\
&\leq C(T)
\end{align*}
The last constant is independent with $n$.
\end{proof}

The following theorem is our main result. 
\begin{theorem}\label{Compact}
If $u_0\in C^+_{tem}$, $u_0(x)=0$ $(x\geq 0)$, $u_t(x)\in C(\R_+,C_{tem}^+)$ is the solution to equation (\ref{eq2}), then $\exists N(\omega)$, such that $u(t,x,\omega)=0$ $\forall x\geq N(\omega)$.
\end{theorem}

Before proving the main theorem, we need some simple estimates. Define 
$$P_tf(x)=(4\pi t)^{-1/2}\int_{\R} e^{-|x-y|^2/4t}f(y)dy,$$
then 
$$\mathcal{F}\left(\int_0^\infty t^{\delta/2-1}e^{-t} P_t f dt\right)(\xi)=\mathcal{F}{f}(\xi)\int_0^\infty t^{\delta/2-1}e^{-t}e^{-t|\xi|^2}dt=\mathcal{F}{f}(\xi)\Gamma (\delta/2)(1+|\xi|^2)^{-\delta/2}$$
Hence 
$$(1-\Delta)^{-\delta/2}f=c(\delta)\int_0^\infty t^{\delta/2-1}e^{-t}P_tfdt$$
Define 
$$R_\delta(x)=c(\delta)\int_0^\infty t^{\delta/2-1}(2\pi t)^{-1/2}e^{-t}e^{-|x|^2/4t}dt=c(\delta)\int_0^\infty t^{\delta/2-3/2}e^{-t}e^{-|x|^2/4t}dt$$
Suppose  $\delta<1$, if $|x|<<1$ then 
$$\int_0^\infty t^{\delta/2-3/2}e^{-t}e^{-|x|^2/4t}dt\leq C |x|^{\delta-1}\int_0^\infty s^{-\delta/2-1/2}e^{-s}ds\leq C |x|^{\delta-1}$$
If $|x|>>1$,  
$$\int_0^\infty t^{\delta/2-3/2}e^{-t}e^{-|x|^2/4t}dt\leq C e^{-|x|}\int_1^\infty t^{\delta/2-3/2} dt+C e^{-|x|^2/2}\int_0^1 e^{-1/2t}dt\leq C e^{-|x|}$$
Hence, $R_{\delta}\in L^p(\R)$ with $p<1/(1-\delta)$; \\
Suppose $\delta=1$, if $x<<1$, then 
$$R_1(x)=c(\delta)\int_0^\infty t^{-1}e^{-t}e^{-|x|^2/4t}dt\leq C \z(\int_{0}^{|x|}t^{-1}e^{-|x|/t}dt+\int_{|x|}^\infty t^{-1}e^{-t}dt \y)\leq  -C\log|x|$$
and $R_1\leq C e^{-c|x|}$ when $|x|\rightarrow \infty$. Hence $R_1\in L^p(\R)$ $(p<\infty)$;\\
Suppose $\delta>1$, then $R_\delta$ is bounded and not greater than $C e^{-|x|}$ when $|x|\rightarrow \infty $. Hence $R_\delta\in L^p(\R)$ $(p\leq \infty )$. \\
By the same argument we have $R'_{\delta+1}\in L^p(\R)$ with $p(1-\delta)<1$. 

\begin{lemma}\label{le5}
Suppose $u\in C(\R_+;C^+_{tem})$ is a solution to \eqref{eq2} with $u_0(x)=0$ for any $x>0$, then 
$$\E \sup_{t\leq T} \left(\int_0^\infty u_t(x)dx\right)^2\leq Ce^{CT} ; ~~~~\all\int_0^\infty u^2_t(x)dx\leq Ce^{CT}. $$
\end{lemma}
\begin{proof}
Choose
\[
\phi_n(x)=
\begin{cases}
0 &{x\leq 0 ~\mbox{or} ~x>n}\\
\frac{1}{2}[1+\sin \pi(x-\frac{1}{2})]&{0<x\leq1}\\
1 &{1<x<n-1}\\
\frac{1}{2}[1+\cos\pi (x-n+1)] &{n-1<x\leq n}\\
\end{cases}
\]
For convenience we omit the subindex $n$. And all the estimates below are independent with $n$.\\
By definition
$$(u_t,\phi)=\int_0^t (u_s,\phi'')ds+\int_0^t (u_s^\gamma\phi, \varphi_k)d\beta_s^k+\int_0^t (u_sh_k,\phi)dw^k_s$$
By Doob's inequality
\begin{equation}\label{eq4}
\begin{split}
\E\sup_{t\leq T}(u_t,\phi)^2 &\leq C \z\{\E\int_0^T \sum_k (u^\gamma_t\phi,\varphi_k)^2 dt+\E\int_0^T \sum_k (u_t\phi,h_k)^2 dt +\E\z(\int_0^T|(u_t,\phi'')|dt\y)^2\y\}\\
&\leq C\z\{ \E\int_0^Tdt\int_\R u_t^{2\gamma}(x)\phi^2(x)dx+\E\int_0^T dt\z(\int_\R u_t(x)\phi(x)dx\y)^2+\E\z(\int_0^T(u_s,|\phi''|)dt\y)^2\y\}\\
&\leq C \z\{a(T)+\E\int_0^Tdt\int_\R [u_t(x)\phi(x)+u_t^2(x)\phi^2(x)]dx+\E\int_0^T dt\z(\int_\R u_t(x)\phi(x)dx\y)^2\y\}\\
&\leq C \z\{1+\E\int_0^Tdt\int_\R u^2_t(x)\phi^2(x)dx+\int_0^T \E(u_t,\phi)^2dt\y\}
\end{split}
\end{equation}
Let $v_t=u_t\phi$, $v_t$ satisfies the following equation
$$\partial_tv=\Delta v+\kappa+\phi^{1-\gamma}v^\gamma \varphi_k\dot{\beta}_t^k+vh_k\dot{w}_t^k$$
where $\kappa=-2(u\phi')'+u\phi''$.
\begin{align*}
\|(1-\Delta)^{-1}\kappa\|_{\mathbb{L}^2(T)}^2&=\E\int_0^T\int_\R |R_2*\kappa_t(x)|^2dxdt\\
&\leq C \z\{\all\int_\R|R_2'*(u_t\phi')(x)|^2dx+\all\int_\R|R_2*(u_t\phi'')(x)|^2dx\y\}\\
&\leq C \z\{\all \z(\int_\R|u_t\phi'(x)|dx\y)^2 +\all \z(\int_\R |u_t\phi''(x)|dx\y)^2\y\}\\
&\leq C \z\{\all \int_0^1 u_t(x)^2dx+\all \int_{n-1}^n u_t(x)^2dx\y\}\\
&\leq a_2(T)
\end{align*}
\begin{align*}
\|(1-\Delta)^{-1/2}\phi^{1-\gamma}v^\gamma\varphi_k\|_{\mathbb{L}^2(l^2)}^2&=\all\int \sum_k|R_1*\phi^{1-\gamma}v^\gamma\varphi_k|^2dx\\
&\leq C \all\int_\R dx\int_\R R_1^2(x-y)v^{2\gamma}(y)dy\\
&\leq C\all\int_\R v_t^{2\gamma}(x)dx\\
&\leq C K_\epsilon\all\int_\R v_t(x)dx+C \epsilon\all\int_\R v_t^2(x)dx\\
&\leq C K_\epsilon\all \z( \int_\R v_t(x)dx\y)^2+C \epsilon\all\int_\R v_t^2(x)dx
\end{align*}
\begin{align*}
\|(1-\Delta)^{-1/2}vh_k\|_{\mathbb{L}^2(l^2)}^2&=\all\int_{\R}\sum_k|R_1*v_th_k(x)|^2dx\\
&=\all \int_{\R} dx \left[\sum_{k}\left(\int_{\R} R_1(x-y)v_t(y)h_k(y)dy\right)^2\right]\\
&\leq \all\int_{\R} dx \left\{\int_{\R} \left[ |R_1(x-y)v_t(y)|^2\sum_k h_k^2(y)\right]^{1/2}dy\right\}^2\\
&\leq C \all  \|R_1*v_t\|_2^2\leq C \all \|R_1\|_2^2\|v_t\|_1^2\\
&\leq C \all \z(\int_\R v_t(x)dx \y)^2
\end{align*}
By \cite[Theorem 5.1]{krylov1999analytic},
\begin{align*}
\|v\|_{\mathfrak{L}^2(T)}^2&\leq C\z\{\|(1-\Delta)^{-1}\kappa\|_{\mathbb{L}^2(T)}^2+\|(1-\Delta)^{-1/2}\phi^{1-\gamma}v^\gamma\varphi_k\|_{\mathbb{L}^2(l^2)}^2 +\|(1-\Delta)^{-1/2}vh_k\|_{\mathbb{L}^2(l^2)}^2\y\}\\
&\leq C\z[1+K'_\epsilon\all \z(\int_\R v_t(x)dx \y)^2\y]+C\epsilon\all\int_\R v_t^2(x)dx
\end{align*}
Choose $\epsilon$ small, such that $C\epsilon\leq 1/2$. Since $\|v\|_{\mathfrak{L}^2(T)}^2\geq \all\int_\R v_t^2(x)dx$, we have
\begin{equation}\label{eq5}
\all\int_\R v_t^2(x)dx\leq C\z[1+K'_\epsilon\all \z(\int_\R v_t(x)dx \y)^2\y]
\end{equation}
Combining (\ref{eq4}),(\ref{eq5}) we get
$$\E\sup_{t\leq T}\z(\int_\R v_t(x)dx\y)^2\leq C \z[1+\all\z(\int_\R v_t(x)dx\y)^2\y]$$
Using Gronwall's inequality, 
$$\E\sup_{t\leq T}\z(\int_\R v_t(x)dx\y)^2\leq Ce^{CT}$$
Since our estimates independent with $n$, we can let $n\rightarrow \infty$, we get
$$\E\sup_{t\leq T} \left(\int_0^\infty u_t(x)dx\right)^2\leq Ce^{CT}$$
$$\all\int_0^\infty u^2_t(x)dx\leq C \z[1+\E\sup_{t\leq T}\z(\int_0^\infty u_t(x)dx\y)^2\y]\leq Ce^{CT}$$
\end{proof}

\begin{lemma}
Suppose $u\in C(\R_+;C^+_{tem})$ is a solution to \eqref{eq2} satisfying $u_0(x)=0$ on $\R_+$. Then 
\begin{equation}\label{Holder u}
\|u\|_{C^\alpha([0,T]\times\R_{+}})<\infty\,\,\,\,  a.s. ,
\end{equation}
for some $\alpha\in(0,1)$ and 
\begin{equation} \label{Int xu}
\E\sup_{t\leq T} \int_0^\infty xu_t(x)dx<\infty 
\end{equation}
\end{lemma}
\begin{proof}
Let
\[
\zeta(x)=
\begin{cases}
0 &{x\leq 0 }\\
\frac{1}{2}[1+\sin \pi(x-\frac{1}{2})]&{0<x\leq1}\\
1 &{x>1}\\
\end{cases}
\]
Like the proof of Lemma \ref{le5} we define $v_t=u_t\zeta$, then $v_t$ satisfies the equation
$$\partial_tv=\Delta v+\kappa+\zeta^{1-\gamma}v^\gamma \varphi_k\dot{\beta}_t^k+vh_k\dot{w}_t^k;~~~~ v_0=0$$
where $\kappa=-2(u\zeta')'+u\zeta''$. Let 
$$T_n=n\land \inf\z\{t\geq 0: \int_0^t \z(\int_\R v_s(x)dx\y)^pds\geq n\y\}.$$ 
then $T_n\rightarrow \infty$ a.s.. Since $R_2$ behaviors like $-\log|x|$ near the original and decreases exponentially, as $|x|\rightarrow \infty$ as before we can prove
$$\|(1-\Delta)^{-1}\kappa\|_{\mathfrak{L}^p(T)}^p<\infty\,\,\,\,(\forall p\in \mathbb{N}).$$
We claim  for any $p\geq 2$, 
\begin{align}\label{vpnorm} \E\int_0^T\|v_t\|_{p}^{p}dt<\infty.\end{align}
Let $p_k=2/\gamma^k$, if we have
$$\E\int_0^T\|v_t\|_{p_k}^{p_k}dt<\infty.$$
then 
\begin{align*}
\|(1-\Delta)^{-1/2}\zeta^{1-\gamma}v^\gamma\varphi_k\|_{\mathbb{L}^{p_{k+1}}(l^2)}^{p_{k+1}}&\leq C \all \int_\R dx\z[\int_\R R_1^2(x-y)v_t^{2\gamma}(y)dy\y]^{p_{k+1}/2}\\
&\leq C \all\int_\R v_t^{p_k}(x)dx<\infty
\end{align*}

\begin{align*}
\|(1-\Delta)^{-1/2}vh_k\|_{\mathbb{L}^{p_{k+1}}(l^2)}^{p_{k+1}}&\leq C \all \int_\R dx \z\{\sum_k \z[\int_\R R_1(x-y)v_th_k(y)\y]^2\y\}^{p_{k+1}/2}\\
&\leq C \all \int_\R dx \z[\int_\R R_1(x-y)v_t(y)dy\y]^{p_{k+1}}\\
&\leq C \all \z(\int_\R v_t(x) dx\y)^{p_{k+1}}<\infty
\end{align*}
Hence, if \eqref{vpnorm} holds, by \cite[Theorem 5.1]{krylov1999analytic}, we obtain  
$$\E\int_0^T\|v_t\|_{p_{k+1}}^{p_{k+1}}dt\leq \|v\|_{\mathbb{L}^{p_{k+1}}(T)}^{p_{k+1}}<\infty$$
Hence, for any $p\geq2$, inequality \eqref{vpnorm} holds. 
Like the argument in the Lemma 1.5 of \cite{krylov1997result}, choose $1/2<\delta<1$, for any $p\geq 2$, 
\begin{align*}
\left(\int_{\R} |R_{\delta+1}*\kappa_t(x)|^pdx \right)^{1/p}&\leq C \left(\int_{\R} |R_{\delta+1}'*(\phi'u_t)(x)|^pdx\right)^{1/p}+ C\left(\int_{\R} R_{\delta+1}*(\phi''u_t)(x)dx\right)^{1/p}\\
&\leq C (\|R'_{\delta+1}\|_{p}+\|R_{\delta+1}\|_p) \int_{[0,1]\cup[n-1,n]}u_t(x)dx
\end{align*}
Hence 
\begin{align*}
\all \int_\R |R_{\delta+1}*\kappa_t(x)|^pdx &\leq C \all \left(\int_{[0,1]\cup[n-1,n]} u_t(x)dx\right)^p<\infty
\end{align*}

\begin{align*}
\|(1-\Delta)^{-\delta/2}\phi^{1-\gamma}v^\gamma\varphi_k\|_{\mathbb{L}^{p}(l^2)}^{p}&\leq C \all \int_\R dx\z[\int_\R R_\delta^2(x-y)v_t^{2\gamma}(y)dy\y]^{p/2}\\
&\leq C \|R_\delta\|_2^{p/2}\all\int_\R v_t^{\gamma p}(x)dx<\infty
\end{align*}

\begin{align*}
\|(1-\Delta)^{-\delta/2}vh_k\|_{\mathbb{L}^p (l^2)}^{p}&\leq C \all \int_\R dx \z\{\sum_k \z[\int_\R R_\delta(x-y)v_th_k(y)\y]^2\y\}^{p/2}\\
&\leq C \all \int_\R dx \z[\int_\R R_\delta(x-y)v_t(y)dy\y]^{p}\\
&\leq C \|R_\delta\|_1^p \all \|v_t\|_p^p <\infty
\end{align*}
By \cite[Theorem 5.1]{krylov1999analytic}, we have $v\in \mathcal{H}_p^{1-\delta}$.  By choosing for instance $\delta=0.6$, $p=33$, we have $v\in \mathcal{H}_p^{1-\delta}(T)$, and by Sobolev's embedding theorem, $C^{1/10}([0,T]\times \R)\subset \mathcal{H}_p^{1-\delta}(T)$ we get  $$\|u\zeta\|_{C^{1/10}([0,T]\times\R)}<\infty\,\,\,\,a.s.$$
Now we prove \eqref{Int xu}. 
Choose $\eta_n(x)\in C_c(\R)$ and $supp\,\eta_n\in \R_+$, $\eta_n(x)=x$ when $x\in [1,n]$ and $\sup_{x,n}|\eta''_n(x)|<\infty$. 
 \begin{align*}
 0\leq\int_1^n xu_t(x)dx \leq& \int_{\R} \eta_n(x)u_t(x)dx\\
 =&\int_0^t \int_{\R}u_s(x)\eta''_n(x)dxds+M^n_t
 \end{align*}
 \begin{align*}
 [M^n]_t=&\int_0^t u^{2\gamma}_s(x)\eta^2_n(x)dxds+\sum_k \int_0^t ds \left(\int_{\R} u_s(x)h_k(x)\eta_n(x)dx\right)^2\\
 \leq& C \int_0^t\left(\int_\R u_s(x)dx\right)^2ds+\int_0^t\int_\R u^2_s(x) dxds\in L^1(\Pro)
 \end{align*}
 Hence $M^n_t$ is a martingale, taking expectation and let $n\rightarrow \infty$, we get
 $$\E \sup_{t\leq T} \int_0^\infty  xu(t,x)dx\leq C \E\sup_{t\leq T} \int_0^\infty u_t(x)dx<\infty $$
\end{proof}

The proof of Theorem \ref{Compact} follows the idea of \cite{krylov1997result}, we present here for reader's convenience. 
\begin{proof}
[Proof Of Main Theorem]
We follow the proof in \cite{krylov1997result}. 

\emph{Step1.} For $\psi\in C_c (\R)$, if $\psi''=\nu$ is a finite measure on $\R$, then equation \eqref{eq2} also holds. (see \cite[Lemma 3.1]{krylov1997result})
 
\emph{Step 2.} On the set $\{\omega : \int_0^T u_s(0,\omega)ds=0\}$, $u(t,x)=0$ $\forall x>0, t\in[0,T]$.

Let $\psi_n=x_{+}\phi(x/n)$,
\begin{align}\label{Mn}
0\leq \int_0^\infty \psi_n(x)u_t(x)dx=\int_0^t u_s(0)ds+\int_0^t \int_0^\infty u_s(x)\psi''_n(x)dxds+M^n_t
\end{align}
where $M^n_t$ is a local martingale with 
$$[M^n]_t=\int_0^t \int_0^\infty \psi_n^2(x)u^{2\gamma}_s(x)dxds+\sum_k \int_0^t ds \left(\int_\R\psi_n(x)h_k(x)u_s(x)dx\right)^2\in L^1(\Pro).$$ 
Hence $M^n_t$ is a martingale, so $\E|M^n_\tau|=2\E(M_\tau^n)^-$ for any bounded stopping time $\tau$. Let $$V^n_t=\int_0^t u_s(0)ds+\int_0^t \int_0^\infty u_s(x)\psi''_n(x)dxds.$$
Using \eqref{Mn}, 
$$(M^n_t)^-\leq V^n_t\leq C \z\{\int_0^tu_s(0)ds+\frac{1}{n}\int_0^t\int_0^\infty u_s(x)dxds+\frac{1}{n^2}\int_0^t\int_0^\infty xu_s(x)dxds\y\}$$ 
Hence for any bounded stopping time $\tau$, 
$$\E|M^n_\tau|=2\E(M_\tau^n)^-\leq 2\E V^n_\tau$$
By the generalized Ito's inequality, we get for any $0<\alpha<1$ and any bounded stopping time $\tau$, 
\begin{align*}
\E \left(\int_0^\tau \int_0^\infty \psi_n^2(x)u^{2\gamma}_s(x)dxds\right)^{\alpha/2}\leq& \E[M^n]_\tau^{\alpha/2} \leq C \E (\sup_{s\leq \tau}|M_s^n|)^\alpha\\
\leq& C \E (V^n_\tau)^\alpha\leq C\E V^n_\tau\\
\leq &C\Big\{ \E\int_0^\tau u_s(0)ds+ \frac{1}{n} \E \int_0^\tau\int_0^\infty u_s(x)dxds\\
&+\frac{1}{n^2}\E\int_0^\tau\int_0^\infty xu_s(x)dxds\Big\}
\end{align*}
Let $n\rightarrow \infty$, we get 
$$\E \left(\int_0^\tau \int_0^\infty x^2(x)u^{2\gamma}_s(x)dxds\right)^{\alpha/2}\leq C  \E\int_0^\tau u_s(0)ds$$
Now let $\tau=\inf\{t: \int_0^t u_s(0)ds>0\}$, we obtain the conclusion.

\emph{Step 3.}Following the proof of Lemma 2.1 of \cite{krylov1997result} one can show: If $\gamma\in [\frac{1}{2},1)$, then for any $p,q>0$, $0<r\leq 1$, $0<\alpha<1$ there exists a point $x\in [r,2r]$ such that
$$\Pro\z(\int_0^T u_s(x)ds\geq p\y)\leq \Pro\z(\int_0^T u_s(0)ds\geq q\y)+Cr^{-3\alpha/2}\z(\frac{q}{p^\gamma}\y)^{\alpha}$$
and here $C$ is independent with $p,q,r$. 

\emph{Step 4.} Now the Theorem can be prove just as Theorem 1.7 of \cite{krylov1997result}.

\end{proof}

\bibliographystyle{plain}

\bigskip
{\bf Guohuan Zhao}

School of Mathematical Sciences, Peking
University,

Beijing, 100871, P.R. China.

Email: zhaogh@pku.edu.cn

\end{document}